%% file: AtypicalHodgeLoci.tex
\documentclass{amsproc}

\input{pgmacs.tex}

\usepackage{amsmath}
 \usepackage{enumerate}
  \usepackage[mathscr]{euscript}
 \usepackage{setspace}

    \newcommand\HL{\mathop{\rm HL}\nolimits}
   \newcommand\pos{\text{\rm pos}}
   
  \newcommand\bsp[1]{\begin{split} #1 \end{split}}
  
  \newcommand\beb{\begin{enumerate}[$\bullet$]}
 \newcommand\eeb{\end{enumerate}}

\newcommand\vspth{\vspace*{-2pt}}
   \newcommand\cV{\mathcal{V}}
   
   \newcommand\bul{\bullet}

  \renewcommand\theenumi{\roman{enumi}} 
  \newcommand\bull{\bullet}

  \theoremstyle{mytheo}
  
  \newtheorem{mainth}[equation]{Main Theorem}
   \theoremstyle{note}

\newtheorem{Definb}[equation]{Definition}
  \newtheorem*{definb}{Definition}

  \renewcommand{\theequation}{\thesection.\arabic{equation}}

   \DeclareFontFamily{U}{mathx}{}
\DeclareFontShape{U}{mathx}{m}{n}{<-> mathx10}{}
\DeclareSymbolFont{mathx}{U}{mathx}{m}{n}
\DeclareMathAccent{\widecheck}{0}{mathx}{"71}

  \renewcommand\CD{\widecheck{D}}
   \renewcommand\gg{\mathfrak{g}}

\newcounter{Cequ}

\begin{document}
      \begin{spacing}{1}
      \title{Atypical Hodge Loci}
      \author{Phillip Griffiths}
      \address{Institute for Advanced Study, Einstein Drive, Princeton, NJ 08540 and Department of Mathematics, University of Miami, Coral Gables, FL 33124}
      \email{pg@ias.edu}
     \thanks{This paper is   based on  \cite{BKU} and related works given in the references in that work, and on   discussions with Mark Green, Matt Kerr and Colleen Robles.  In particular I would like to thank Colleen Robles for help with the technical arguments used in the proof of the main result.}

\subjclass[2020]{Primary 14, 22, 42, 53}

\date{}

\begin{abstract}
In the recent works of a number of people there has emerged a beautiful new perspective on the arithmetic properties of Hodge structures. A central result in that development appears in a paper by Baldi, Klingler, and Ullmo. In this expository work we will explain that result and give a proof. The main conceptual step is to formulate Noether-Lefschetz loci in terms of intersections of period images with Mumford-Tate subdomains of period domains. The main technical step is to use the alignment of the Hodge and root space decompositions of the Lie algebras of the associated groups and from it to use the integrability conditions associated to a Pfaffian PDE system.  These integrability conditions ``explain'' the generally present excess intersection property associated to the integral varieties of a pair of Pfaffian exterior differential systems.
\end{abstract}

\maketitle
       
       \tableofcontents
       
 Hodge theory may be said to include three types of applications to algebraic geometry; namely to 
       \beb
       \item the geometric properties of algebraic varieties;
       \item the topology of algebraic varieties;
       \item the arithmetic properties of algebraic varieties.
       \eeb
       This paper  will discuss a topic in the third item.
       
       Recently there has been a flurry of activity concerning the arithmetic properties, especially functional transcendence, of periods of integrals of algebraic differential forms.  A summary of this theory appears in the proceedings of the ICM 2020 paper \cite{K1}.  In this expository paper we will explain and discuss a proof of one of  the main theorems in the work \cite{BKU}.  This beautiful result plays a central role in the recent developments in functional transcendence theory.  In the proof in loc.\ cit. of this result there are two main Hodge theoretic components, one conceptual and the other Lie theoretic.  The conceptual one is to use intersections of period images with Mumford-Tate subdomains instead of the more traditional Noether-Lefschetz loci.  The Lie theoretic one is to use the compatible alignments of the Hodge and root space decompositions of the Lie algebras of the  automorphism groups of the  period domain and Mumford-Tate subdomain to produce the non-transversality of the integrability conditions associated to a pair of Pfaffian exterior differential systems.
       
       This paper grew out of a talk given at the Regulators V Conference held June 3--15, 2024, at Pisa.  A main purpose of both the talk and this paper is to help draw attention to  some of the Hodge theoretic aspects of the developments in arithmetic algebraic geometry that are discussed in \cite{K1}, \cite{BKU} and in a number of subsequent works including \cite{BKT}, \cite{BBKT}, \cite{C}, \cite{HP}, \cite{KO}, \cite{U}.  Hopefully it will also help serve  as an invitation for Hodge theorists and others to  this  very rich   and currently very active development.
       \section{Introduction}\label{sec1}
       We will be concerned with the question
       \begin{quote}
       \emph{What can one say about Hodge loci?}
       \end{quote}
       Specifically, we consider the situation where
       \beb
       \item $B$ is a smooth, connected quasi-projective variety;
       \item $\V\to B$ is the local system underlying a variation of \phs\ of weight $n$;
       \item $\HL(B)$ is the set of points $b\in B$ where there are more Hodge classes in the tensor algebra $\V^\otimes_b:= \opplus^{k,\ell}( \V_b^{\otimes^k}\otimes V_b^{\ast\otimes \ell} )$ than there are at a very general point of $B$.   
           \eeb
       A central question is 
       \begin{quote} \emph{What can we say about $\HL(B)$?}
       \end{quote}
            In \cite{CDK} it is proved that $\HL(B)$ is a countable union of algebraic varieties.
Very informally stated, denoting by $\HL(B)_{\pos}$ the positive dimensional components of $\HL(B)$ the result of \cite{BKU} is
         \begin{equation}\lab{1.1} \bmpt{4.5}{\emph{For $n\geqq 3$ and aside from  exceptional cases, every irreducible component of $\HL(B)_\pos$ has less than the expected codimension.\footnotemark}}                \end{equation}       
              \footnotetext{For subvarieties $A,B$ of a variety $C$ the expected codimension in $C$ of $A\cap B$ is $\codim_C A+\codim_C B$; i.e., that which would be obtained if the intersection were transverse.}
              \noindent
               That means that  if there are Hodge classes that vary in a positive dimensional family, then there are strictly more of these than suggested by a standard dimension count. As explained below the reason for this will be a consequence of  the integrability conditions arising from the differential constraint of the variation of Hodge structure.  In calculating these a key ingredient is 
                 properties    relating the Hodge and root space decompositions of  a semi-simple Hodge Lie algebra  \cite{R}.
      
      Although we shall not discuss it, in \cite{BKU} it is proved that if \eqref{1.2} is satisfied, then $\HL(B)_\pos$ is a \emph{finite} union of irreducible algebraic varieties.

            The \emph{geometric case} is when the variation of Hodge structure arises from the cohomology along the fibres of a smooth family $\cX\to B$ of projective varieties.  In this case, assuming the Hodge conjecture the result gives that if there are algebraic cycles $Z_b \subset X_b$ whose cohomology classes are not found on a general $X_b$ and that non-trivially vary with~$b$, then there are strictly more such cycles than one naively expects to find.

               An interesting point is that whereas in general integrability conditions \emph{decrease} the expected dimension of the space of solutions to a single system of differential equations, for the pair of EDS's in the case at hand this dimension is actually \emph{increased}.  The mechanism behind this will be illustrated in Example 2 below.
       
           With the notation to be explained below a sufficient condition for \eqref{1.1} to hold is
           \begin{equation}\lab{1.2}
           \gg^{-k,k} \ne 0\hensp{for some} k\geqq 3.\end{equation}
           In \cite{BKU} it is shown that these conditions hold for smooth hypersurfaces $X\subset \P^{n+1}$ where $n\geqq 3$, $d\geqq 5$ and $(d)\ne (4.5)$.  We will give the argument for this in Section~\ref{sec5}.
           
            In Section \ref{sec5} we will also recall the definition of the coupling length $\zeta(\AA)$  where $\AA\subset \gg^{-1,1}$ is the image of the differential of the period mapping at a general point (\cite{VZ}).  There  we will show that
           \begin{equation}\lab{1.3}
           \zeta(\AA)\geqq 3\implies \text{\eqref{1.2}}.\end{equation}
    Examples where the coupling length $\zeta(\AA)\geqq 3$ include   the moduli   of Calabi-Yaus of dimension $\geqq 3$ whose Yukaya coupling is non-zero.  
           
           The case 
           \begin{equation}\lab{new1.3}
           \gg^{-k,k}=0 \hensp{ for }k\geqq 3\end{equation} includes the \emph{classical case} where the period domain is Hermitian symmetric. 
              The case $\gg^{-2,2}\ne 0$ but $\gg^{-k,k}=0$ for $k\geqq 3$ includes the case of weight $n=2$ Hodge structures; here the differential constraint is non-trivial but the corresponding integrability conditions do not enter into the dimension count.
          In this situation classical results in \cite{G1}  \cite{V1}, and \cite{V2} give that in many cases Noether-Lefschetz loci are analytically dense in moduli.      The classical case and the case of weight $n=2$ are   discussed in detail in \cite{BKU}.
    
           The  conditions to have \eqref{new1.3}
        may be expressed as a non-linear PDE system for the variation of Hodge structure.  At least in principle the E. Cartan theory of exterior differential systems gives an algorithmic procedure for determining the space of formal solutions to this system (\cite{BCGGG}).  It would seem of interest to carry this out in the present situation.

           \section{Two examples}\label{sec2}\setcounter{equation}{0}
           We consider the geometric case where $\cX\to B$ is a family of surfaces with $X_{b_0}=X$ and $\la\in \Hg^1(X)_{\prim}$ is a primitive Hodge class.
           
           \begin{equation} \lab{2.1}
           \bmpt{4.25}{{\bf Q:} \emph{How many conditions is it for $\la$ to vary with $X$ as a Hodge class?}}\end{equation}
           We restrict to a neighborhood $U$ of $b_0\in B$ so that $\la\in H^2(X_b)_\prim$ is well defined.  Setting 
    \[
    \NL_\la = \{b\in U:\la\in \Hg^1(X_b)\}\]
    we are asking \emph{What is the codimension of $\NL_\la$ in $U$?}  For this there is the classical estimate
    \begin{equation}\lab{2.2}
    \codim_U\NL_\la \leqq h^{2,0}(X).\end{equation}
    This bound is achieved; e.g., for smooth surfaces $X\!\subset\! \P^3 $ with $d=\deg X\!\geqq\! 4$ (cf.\ \cite{G1} and \cite{G2}).\footnote{A general treatment of Noether-Lefschetz loci given in \cite{G3} and in \cite{V1}, \cite{V2}.}  In this case one also has
    \begin{equation}\lab{2.3}  
    d-3\leqq \codim_B \NL_\la\end{equation}
 with equality holding only for surfaces containing a line.  It is also proved in loc.\ cit.\ that $\HL(B)$ is analytically dense where $B$ is the moduli space.
     
    For the second example taken from \cite{GG} and \cite{GGK} we let $\cX\to B$ be a family of 4-folds with $\la\in \Hg^2(X)_\prim$ a primitive Hodge class for $X=X_{b_0}$ and ask the same question.  The analogue of \eqref{2.2} is 
    \begin{equation}\lab{2.4}
    \codim_B\NL_\la \leqq h^{4,0}(X) +h^{3,1}(X).\end{equation}
    However due to transversality of the period mapping we have that for the first order variation of $X$ in a direction $\theta\in T:= T_{b_0}B$  the product $\theta\cdot\la\in H^{1,3}(X)$ 
  and therefore for any $\om\in H^{4,0}(X)$
      \begin{equation}\lab{2.5}
    \lra{\om,\theta\cdot \la}=0.\end{equation}
    Thus a refinement of \eqref{2.4} is
    \begin{equation}\lab{2.6}
    \codim_B\NL_\la \leqq h^{3,1}(X).\end{equation}
    We refer to the right-hand side of \eqref{2.6} as the \emph{naive expected codimension} of $\NL_\la$ in $B$.
    
    At this juncture a new consideration enters.  Setting
    \begin{equation}\lab{2.7}\bsp{
    T_\la&:= \{\theta\in T: \theta\cdot \la=0\hensp{in}H^{1,3}(X)\},\\
    \sig(\la)&=\text{Image}\{T_\la\otimes H^{4,0}(X)\to H^{3,1}(X)\}}\end{equation}
    we have
    \begin{equation}\lab{2.8}
    \codim_B\NL_\la\leqq h^{3,1}(X)-\dim \sig(\la).\end{equation}
    
    \begin{proof}
    For $\theta\in T_\la$ and any $\theta'\in T$, $\om\in H^{4,0}(X)$
    \begin{align*}
    \lra{\theta\cdot \om,\theta'\la}&= -\lra{\om,\theta\theta'\la}\\
    &= -\lra{\om,\theta'\theta\la}\\
    &=0\end{align*}
    where the second step follows from the integrability condition $\theta\theta'=\theta'\theta$ arising from transversality.
   \end{proof}
   
   Assuming that the map defining $\sig(\la)$ in \eqref{2.7} is non-zero we see that due to integrability the actual codimension of $\NL_\la$ is strictly less than the naive expected codimension.
   
   One may show that the estimate \eqref{2.8} is sharp; e.g., by taking $X\subset \P^5$ a hypersurface of degree 6 containing a 2-plane $\La$ and for $\la\in\Hg^2(X)_\prim$ the primitive part of the class of $\La$ (cf.\ \cite{GG}).
    
    Finally there is a discussion of Noether-Lefschetz loci in \cite[\S III.C]{GGK}.  A general version of the second example above is given there in III.C.5.
        \section{The main result}\label{sec3}
    (i) \emph{Hodge structures and \mtg s}.\footnote{A general reference for Hodge theory is \cite{CM-SP}.  For \mtg s we will generally follow the presentation and notations given in \cite{GGK}.}  A \emph{\phs}\ of weight $n$ is given by the data $(V,Q,\fbul)$ where
    
    \beb
    \item $V$ is a $\Q$-vector space and $Q\!:\!V\otimes V\to \Q$ is a non-degenerate bilinear form with $Q(u,v)=(-1)^n Q(v,u)$;
    \item $F^n\subset F^{n-1}\subset\cdots\subset F^0\subset V_\C$ is a \emph{Hodge filtration} satisfying
  \vspth$$
    F^p\oplus \ol{F}^{n-p+1}\simto V_\C, \quad 0\leqq p\leqq n;  \vspth$$ and
    \item the two Hodge-Riemann bilinear relations are satisfied (we do not need their explicit form).
    \eeb
    Setting 
    \[
    V^{p,q}=F^p\cap \ol F^q\]
    the second condition above is equivalent to the \emph{Hodge decomposition} \setcounter{equation}{0}
    \begin{equation}\lab{3.1}
    V_\C=\oplus V^{p,q},\quad \ol{V^{p,q}}= {V^{q,p}}.\end{equation}
    Using $ Q$ we have an identification
    \begin{equation} \lab{3.2}
    V\cong  V^\ast.\end{equation}
       We will generally omit reference to $Q$, its presence being understood.
       
       A Hodge structure of weight $n$ may equally be given by a homomorphism
       \[
       \vp:S^1 \to \GL(V_\R)\]
       such that its action on $V_\C$ is determined by the \emph{Weil operator}
        \vspth \[
       \vp(i)v= i^{p-q}v,\quad v\in V^{p,q}.  \vspth\]
    
    When the weight $n=2m$ the \emph{Hodge classes} are
    \vspth  \[
    \Hg^m(V)=V^{m,m}\cap V,  \vspth\]
    the rational vectors of type $(m,m)$.  These are the rational classes that are fixed by the Weil operator.
    
  We denote by 
  \vspth  \[
  V^{\otimes} := \opplus^{k,\ell}\lrp{ (\ottimes^k V) \otimes (\ottimes^\ell V^\ast)}  \vspth\]
  the tensor algebra of $V$.    The algebra of Hodge tensors involves tensor products of both $V$ and its dual $V^\ast$.  Using $Q$ we may consider only powers of~$V$.  Although somewhat artificial this  simplifies the notations and involves no loss of generality.  We then denote by
   \vspth\[
  \Hg^\bul(V^\otimes):=\opplus^k \Hg^{k\,n/2}(\otimes^{k}V)   \vspth\]
 the sub-algebra of Hodge tensors.

  \begin{definb}
  The \emph{Mumford-Tate group} $\MT(V)$ is the smallest $\Q$-algebraic sub-group of $\Aut(V,Q)$
  such that $\vp(S^1)\subset \MT(V)(\R)$.
    Equivalently (cf.\ (I.B.1)  in \cite{GGK})
  \begin{equation}\lab{new3.4}
  \MT(V)\subset \Aut(V , Q)\hensp{is the subgroup that fixes}\Hg^\bullet(V^\otimes).\end{equation}   \end{definb}

The \mtg\ is a reductive $\Q$-algebraic group   denoted here  by $G$.  Its Lie algebra
  \[
  \gg\subset\End(V,Q)\]
  is a 
\emph{Hodge Lie algebra}; i.e., it  has a Hodge structure of weight zero with Hodge decomposition
  \[
  \gg_\C = \oplus \gg^{-k,k}\]
  where
  \[
  \gg^{-k} :=\gg^{-k,k}=\{A\in \gg_\C\hensp{such that} A:V^{p,q}\to V^{p-k,q+k}\}. \]

  We note that after possibly passing a finite cover the real Lie group
  \[
  G(\R)=G_1\times\cdots\times G_k\times T\]
  is a product of simple Lie groups with a compact torus.  However we will not have a corresponding product decomposition of $G$.
    
    Let $\gg_\R=\oplus \gg_{i,\R}$ be the decomposition of $\gg_\R$ into simple factors.  Following \cite{BKU} we have the
    \begin{Definb} The \emph{level} $\ell(\gg)$ is the smallest $k$ such that  all $\gg^k_i\ne 0$.
    \end{Definb}
    
    (ii) \emph{Variation of Hodge structure.} This is given by the data $(\V,\cF^\bul;B)$ where
    \beb
    \item $\V\to B$ is a local system over a smooth, connected quasi-projective variety $B$;
    \item $\cF^\bul$ is a filtration of $\cV=\V_\C\otimes_\C\cO_B$ by holomorphic sub-bundles that induce on each 
    $\V_b$ a Hodge structure, and where for $\nabla$ the Gauss-Manin connection corresponding to  $\V\subset \cV$ the transversality condition
    \begin{equation}
    \lab{3.4}
    \nabla \cF^p \subset \cF^{p-1}\otimes \Om^1_B\end{equation}
    is satisfied.
    \eeb
    
    It is understood that there is a horizontal section $Q$ of $(\V_\Q\otimes \V_\Q)^\ast$ that polarizes the Hodge structures.
    
    At each point of $B$ there is an algebra of Hodge tensors and \mtg. Outside of a countable union of proper subvarieties of $B$ these algebras are locally constant.  We denote by $V=\V_{b_0}$ the fibre of $\V$ at such a very general point and by $G\subset \Aut(V,Q)$ the corresponding \mtg\ of the variation of Hodge structure.  
    
    The action of $\pi_1(B,b_0)$ on $V$ induces the monodromy group $\Ga\subset \Aut(V,Q)$.  It is known that $\Ga\subset G$ and in \cite{GGK} there is the general structure theorem III.A.1  (cf.\ Corollary III.A.2) describing their relation.  In this works in order to isolate the central points we will make the assumption
    \begin{equation}\lab{3.5}
    \bmpt{4.25}{$G$ \emph{is a simple $\Q$-algebraic group equal to the $\Q$-Zariski closure $\ol{\Ga^\Q}$ of the monodromy group.}}\end{equation}    
    
    As explained in loc.\ cit.\ the general case may be reduced to the case in which up to a finite cover $G$ and $\Ga$ are products of factors $G_i$ and $\Ga_i$  where either $\ol{\Ga_i^\Q}=G_i$ or $\Ga_i$ is trivial.  The induced variations of Hodge structure corresponding to the later are constant and will not contribute to the end result.
    
    (iii) \emph{Period mappings.} Given $(V,Q)$ the set of \phs s with given Hodge numbers $h^{p,q}=\dim V^{p,q}$ is a homogeneous complex manifold called a period domain.  The set of those \phs s whose \mtg\ is contained in $G$ gives a \emph{Mumford-Tate domain}
    \[
    D=G(\R) /G_0\]
    where $G(\R)$ is the real Lie group associated to $G$ and $G_0$ is a compact subgroup.  Following \cite{BKU} for $\Ga$ an arithmetic subgroup of $G$ the quotient $\Ga\bsl D$ is called a \emph{Hodge variety}.
    
    Associated to a variation of Hodge structure with \mtg\ $G$ there is a period mapping
    \begin{equation}\lab{3.6}
    \Phi:B\to\Ga\bsl D.
    \end{equation}
    It may be assumed that $\Phi$ is proper with image $P\subset\Ga\bsl D$ a quasi-projective variety \cite{BBT}.\footnote{In the final Reprise section of this paper we will comment on the proof of this result; this argument brings new ideas into the study of period mappings.}  In order to isolate the essential points in  the following discussion    we will   make the assumption
    \begin{equation}\lab{3.7}
    \text{\emph{$P$ is smooth, and we identify $\Phi(B)=P\subset\Ga\bsl D$.}}\end{equation}
    All of the results discussed below hold without this assumption, and the proofs may be given using standard technical modifications in the general case where \eqref{3.7} may not hold. The one situation that requires extra care is when there is an irreducible subvariety $Z\subset B$ such that $\Phi(Z)$ is contained in the singular locus of $\Phi(B)$ and where the generic Mumford-Tate along $Z$ is smaller than $G$. 
    
    For $b\in B$ and using a lift to $D$ of $\Phi(b)$ the tangent space to $D$ at the point is identified
    with $\gg_\C/F^0\gg_\C$.  Using \eqref{3.4} the differential of the period mapping is
    \begin{equation}\lab{3.8}
    \Phi_\ast :T_b B\to F^{-1}\gg_\C/F^0\gg_\C.\end{equation}
    We may identify the right-hand side of \eqref{3.8} with $\gg^{-1,1}$.  For later use we have (cf.\ \cite{CM-SP})
    \begin{equation}\lab{3.9}
    \Phi_\ast (T_bB):= \mathfrak{A}\text{ is an abelian sub-algebra of }\gg^{-1,1}.\end{equation}
    The ``abelian" is a consequence of the integrability conditions imposed by the transversality property \eqref{3.8}.
    
    It is interesting to note that given a Hodge Lie algebra and abelian   sub-algebra $\AA\subset \gg^{-1,1}$ there is a local VHS with $\AA$ as its tangent space.  Thus there are no ``higher" integrability conditions beyond having an abelian sub-algebra of $\gg^{-1,1}$.
    \smallbreak
    
    (iv) \emph{Hodge loci.} We are interested in proper, irreducible subvarieties $Z\subset B$ along which  the corresponding Hodge structures have extra Hodge tensors.  Equivalently the Lie algebra of the \mtg\ $H$ at a general point of $Z$ should be   strictly contained in  that of  $G$:
    \[
    \hh \subsetneqq  \gg.\]
    A basic observation is
    \begin{equation}\lab{3.10} 
    \bmpt{4.25}{\emph{If $D_H\subset D$ is the $H(\R)$-orbit of a very general point of $Z$, then for $\Ga_H=\Ga\cap H$ we have}
    \[
    \Phi(Z) \subset \Ga_H \bsl D_H.\]}\end{equation}
    
    \begin{definb}[\cite{BKU}]
    If $H\subset G$ is a Mumford-Tate subgroup, then
    \begin{equation}\lab{3.11}
    \Phi^{-1}(\Phi(B)\cap (\Ga_H\bsl D_H))^0\end{equation}
    is a \emph{special subvariety} of $B$.
    \end{definb}
    
    The exponent $^0$ means to take an irreducible component of the intersection.  Recalling the notation $P=\Phi(B)$ we   set
    \begin{equation}\lab{new3.13}
    P_H = (P\cap (\Ga_H\bsl D_H))^0.\end{equation}
    
  The subvariety $Z$ may not be maximal with Mumford-Tate group~$H$.  In order to consider irreducible subvarieties that have extra Hodge classes rather than a particular $Z$ one should use the  intersection~\eqref{3.11}.  
  
  The standard codimension of an intersection estimate gives
    \begin{equation}
    \lab{3.12}
    \codim_{\Ga\bsl D}(\Phi(B)\cap (\Ga_H\bsl D_H))^0 \leqq \codim_{\Ga\bsl D}(\Ga_H\bsl D_H)+
    \codim_{\Ga\bsl D}\Phi(B),\end{equation}
or  in the notation \eqref{new3.13}
    \begin{equation}\lab{3.13}
    \codim_{\Ga\bsl D}P_H  \leqq \codim_{\Ga\bsl D}(\Ga_H\bsl D_H)+\codim_{\Ga\bsl D} P.\end{equation}
    
    \begin{definb} The subvariety $\Phi^{-1}(\Ga_H\bsl D_H)\subset B$ is \emph{atypical} if we have strict inequality in \eqref{3.12}.\end{definb}
    In other words, atypical means there are strictly more Hodge tensors than suggested by the standard    dimension count.
    
    \begin{mainth} \lab{mainth}
    If the variation of Hodge structure has level at least three, then every positive dimensional special subvariety is atypical.
    \end{mainth}
    As will be explained below, the proof will be to show that the condition for $\gg$ to have level at least three   implies  that the integrability conditions arising from transversality are non-trivial for \emph{every} positive dimensional, special  subvariety of $B$.

     \section{Proof of the main result}\label{sec4}
    
    (i) \emph{Sketch of the argument for   Theorem \ref{mainth}.}  The strategy is to assume equality in \eqref{3.12}, or equivalently in \eqref{3.13},
    and from this infer that
    \setcounter{equation}{0}
    \begin{equation}\lab{4.1}
    \hh^- = \gg^-.\end{equation}
    This implies that
    \begin{equation}\lab{4.2}
    D_H=D,\end{equation}
    i.e., the special subvariety is all of $B$.
    
    We let $\wt P\subset D$ be the inverse image in $D$ of $P\subset \Ga\bsl D$ and $\wt P_H= (\wt P\cap D_H)^0$.  Here recall that  the $(\enspace)^0$ means taking an irreducible component of $P\cap D_H$.  From \eqref{3.13} we have
    \begin{equation}\lab{4.3}
    \codim_D(\wt P_H)=\codim_D \wt P+\codim_D D_H.\end{equation}
    Using the assumption \eqref{3.7} we may work  infinitesimally in the tangent space $T_0D\! \cong\! \gg^{-}$ with $T_0 D_H\!\cong\! \hh^{-}$. This gives
    \begin{align*}
    \codim_D \wt P_H &= \dim \gg^- - \dim T_0\wt P_H,\\
    \codim_D\wt P &= \dim \gg^- - \dim T_0\wt P,\\
    \codim_D D_H&= \dim \gg^- - \dim \hh^-.\end{align*}
    Then \eqref{4.3} yields
    \[
   \dim T_0\wt P-\dim T_0\wt P\cap \hh^- = \dim \gg^- -  \dim \hh^-  .\]
    Rewrite this as
    \begin{align*}
    \sum_{p\geqq 2} \dim\hh^{-p} + \codim_{\hh^{-1}}  T_0\wt P_H  = 
    \sum_{p\geqq 2} \dim \gg^{-p} +\codim_{\gg^{-1}} T_0 \wt P.\end{align*}
    Since $\dim \hh^{-p}\leqq \dim \gg^{-p}$ and $\codim_{\hh^{-1}} T_0\wt P_H \leqq \codim_{\gg^{-1}} T_0\wt P$ this forces
    \begin{equation}\lab{4.4}      \hh^{-p}= \gg^{-p}, \qquad p\geqq 2. \end{equation}
    From this we want to conclude \eqref{4.1},   which gives
    \[
    T_0 D_H=T_0 D\]
  and  which then implies \eqref{4.2}.\footnote{We note that   \eqref{4.4} is vacuous in the classical case.}
    
    At this point the idea is that from the condition in \eqref{4.4} the Lie algebra $\cL$   generated by the $\hh^{-k}$, $|k|\geqq 2$ is equal to the Lie algebra generated by the $\gg^{-k}$, $|k|\geqq 2$.  Here $\gg = \oplus \gg^{-k,k}$ is the Hodge decomposition of $\gg_\mathbb{C}$ and $\gg^{-k}:= \gg^{-k,k}$.  If we can show that
    \[
    \ell(\gg)\geqq 3\text{ implies that }    \cL^-=\gg^-,\]
    then we are done. For this it would suffice to show that $\hh^{-1}=\gg^{-1}$.  This would be the case if we   show that the Hodge and root space decompositions of $\gg,\hh$ align, and then that the assumption $\ell(\gg)\geqq 3$ implies that all the simple roots of $\gg$ are also roots of $\hh$.\footnote{For the Lie theoretic aspects of Hodge theory that will be used here we suggest~\cite{R}.}
    
    \medbreak
    (ii) \emph{Proof of Theorem \ref{mainth} under the assumption \eqref{4.6} below.}  We will first give the argument under the assumption
    \begin{equation}\lab{4.6}
    \gg^{\pm} \hensp{bracket generates all of} \gg_\C.\end{equation}
    This assumption may not   be satisfied, but we shall show that in the particular circumstances at hand an   adaptation of the argument assuming it  gives the result.
    
Recall that the semi-simple Lie algebra $\gg$  has a weight zero \phs\ with the Hodge decomposition
    \[
    \gg_\C=\opplus \gg^k,\qquad \gg^{-k}=\gg^{-k,k} = \ol{\gg^{k}}.\]
    that was noted above.  Let $\gg'_\C$ be the complex Lie sub-algebra generated by $\gg^{\pm}$.  Then $\gg'_\C$ is the complexification of a real sub-algebra $\gg'_\R$.  Writing
    \[
    \gg_\R = \gg'_\R\oplus \mathfrak{l}_{\R}\]
    and denoting by $^\bot$ the orthogonal relative to the Cartan-Killing form on $\gg$ it follows that 
    \[
    \mathfrak{l}_R = (\gg'_\R)^\bot \subset \gg^{0}_\R\]
 is an ideal in $\gg$.  Since $\gg$ is semi-simple  and $[\gg^k,\gg^\ell]\subset \gg^{k+\ell}$ it follows that $\mathfrak{l}_R$ is a direct sum of factors $\wt \gg_i$ of $\gg$ where all $\wt \gg^\pm_i = (0)$. This implies 
  that $\exp\mathfrak{l}_R$ acts trivially on $D$ and so we may assume that $\mathfrak{l}_R$ is trivial.
    
    We then  have
     \vspth \beb
    \item $\gg$ is a Hodge Lie algebra;
    \item $\gg^+$ is a nilpotent sub-algebra of $\gg$;
    \item $\gg^+\oplus \gg^0$ is a parabolic sub-algebra with Levi factor $\gg^0$;
    \item the center of $\gg^0$ is contained in  a Cartan sub-algebra $\tt$, and, fixing a Borel subalgebra, we denote by $\beta_{i}$, $ i\in I$, the  positive simple $\tt$-roots of $\gg^+$ with corresponding root spaces~$\gg_{\beta_i}$.\eeb

  \vspth    \begin{Lem}\lab{4.7}
    $\gg^+$ is generated by $\gg^1$ if and only if $\gg_{\beta_i}\subset \gg^1$ for all $i\in I$.
  \vspth    \end{Lem}
    
    \begin{proof}
    To establish the implication $\implies$ we note that since every positive root is $\sum n_i\beta_i$ with $0\leqq  n_i\in \Z$, if some $\gg_{\beta_i}\not\in \gg^1$, then the algebra generated by $\gg^1$ will not contain $\gg_{\beta_i}$.  For the converse see Theorem 3.2.1 (1), and definition 3.1.2, in \cite{CS}
    \vspth  \end{proof}
    
    Now let $\cL\subset \gg$ be the  sub-algebra of $\gg$ generated by the $\hh^{\pm i}$, $i\geqq 2$. It is a sub-Hodge structure and by assumption   \eqref{4.6} is satisfied.

    \vspth In the proof of the following proposition there are three basic facts to keep in mind.

    \begin{enumerate}
        \item[(1)]  We have
        \[
            [\gg_{\alpha}, \gg_{\beta}] = \begin{cases}
                \gg_{\alpha + \beta}, & \alpha+\beta \in \Delta(\gg),\\
                \{0\}, & \alpha + \beta \notin \Delta (\gg).
            \end{cases}
        \]
        \item [(2)]  From $[\gg^k, \gg^{\ell}] \subset \gg ^{k+\ell}$ we see that each $\gg ^k$ is an $\gg ^0$-module.  Since $\tt \subset \gg ^0$, this implies that each $\gg ^k$, is a direct sum of root spaces;  let $\Delta(\gg ^k)$ denote the corresponding roots.  (In the case $k=0$, we must include the ``zero root space'' $\tt \subset \gg^0$.)

        \item[(3)]  Let $\Delta _0(\gg) = \{\beta _1, \dots, \beta_r \} $ denote the simple roots, then $\Delta _0(\gg^k) = \Delta (\gg^k) \cap \Delta _0 (\gg)$ denotes the simple roots $\beta _i$ with $\gg _{\beta _i} \subset \gg^k$.  The $\gg ^0$-module $\gg ^1$ decomposes as a direct sum $\oplus V_{\beta _i}$ of irreducible $\gg _0 $-submodules.  The $V_{\beta _i}$ are indexed by the simple roots $\beta _i \in \Delta _0(\gg ^1)$.  The root space $\gg _{\beta _i}$ is the lowest weight line in $V_{\beta _i}$.  We have $0 \neq \gg _{\alpha} \subset V_{\beta _i}$, if and only if $\alpha \equiv \beta_i$ modulo $\Delta (\gg^0)$.
    \end{enumerate}
    
   \vspth   \begin{Prop}\lab{4.8} Assume that $\gg^1$ generates $\gg ^+$ under the Lie bracket;  the level $\ell (\gg) \geq 3$;  and that $\hh ^k = \gg ^k$ for all $k \geq 2$.  Then $\hh ^+ = \gg ^+$.
   \vspth   \end{Prop}

   \vspth  The proposition is a corollary of the following lemma.

    \vspth    \begin{Lem}\lab{4.9}
        Assume that $\gg ^1$ generates $\gg^+$, and that the level $\ell (\gg) \geq 3$.  Let $\tilde{\gg}$ be the subalgebra generated by $\gg ^{\pm k}, k \geq 2$.  Then $\tilde{\gg}$ contains $\gg ^{\pm 1}$.
    \vspth \end{Lem}
   
    \begin{proof}
    It suffices to show that $\gg ^1 \subset \tilde{\gg}$.  Each $\gg ^k$ is a $\gg _0$-module.  It follows from the Jacobi identity that $\tilde{\gg}$ is a $\gg ^0$-module.  In particular, $\tilde{\gg} = \oplus \tilde{\gg}^k$, where $\tilde{\gg}^k = \gg ^k \cap \tilde{\gg}$.  We have $\tilde{\gg}^k = \gg ^k$ for all $|k| \geq 2$.

    Let $\gg^1 = \tilde{\gg}^1 \oplus V^1$ be the $\gg ^0$-module decomposition.

    We claim that $[\tilde{\gg}^1, V^1] = 0$.  Assume not.  Then there exists $\gg _{\alpha} \subset \tilde{\gg}^1$ and $\gg_{\alpha '} \subset V^1$ so that $0 \neq \gg _{\alpha + \alpha '} = [\gg _{\alpha}, \gg _{\alpha '}] \subset \gg^2 = \tilde{\gg}^2$.  Then $\gg _{\alpha '} = [\gg _{\alpha + \alpha '}, \gg _{-\alpha}] \subset [\tilde{\gg}, \tilde{\gg}] = \tilde{\gg}$, a contradiction.

    So $\gg^2 = [\gg^1,\gg^1]= [\tilde{\gg}^1,\tilde{\gg}^1]+[V^1, V^1]$.  We claim that the sum is necessarily a direct sum
    \[
    \gg ^2 = [\gg^1,\gg^1]= [\tilde{\gg}^1,\tilde{\gg}^1] \oplus [V^1, V^1].
    \]
    To see why note that the root spaces $\gg _{\alpha}$ of $[\gg^1,\gg^1]$ are all of the form $\alpha \equiv \beta _i + \beta _j$ modulo $\Delta (\gg _0)$ with $\beta _i, \beta _j \in \Delta _0 (\tilde{\gg}^1)$.  Likewise the root spaces $\gg _{\alpha '}$ of $[V^1, V^1]$ are all of the form $\alpha ' \equiv \beta ' _i + \beta ' _j$ modulo $\Delta(\gg _0)$ with $\beta '_i, \beta '_j \in \Delta_0(V^1)$.  The claim then follows from the fact that $\Delta _0 (\tilde{\gg}^1)$ and $\Delta _0 (V^1)$ are disjoint sets.

    Next, the Jacobi identity implies
    \[
    \gg^3 = [\gg ^1, [\gg^1, \gg^1]] = [\tilde{\gg}^1,[\tilde{\gg}^1,\tilde{\gg}^1]] \oplus [V^1, [V^1,V^1]].
    \]
    Continuing inductively, we see that
    \[
    \gg ^+ = \tilde{\gg}^+ \oplus V^+,
    \]
    with $\tilde{\gg}^+$ the algebra generated by $\tilde{\gg}^1$, and $V^+$ the algebra generated by $V^1$.  We claim that this forces $V^+ = 0$.  Assuming the claim for the moment, we have $\tilde{\gg}^1 = \gg ^1$.  This establishes the lemma.

    To see why the claim $V^+ = 0$ holds, note that $\Delta(\gg^+)$ must contain the highest root $\tilde{\alpha}$.  The coefficient of $\beta _i \in \Delta _0(\gg)$ in $\tilde{\alpha}$ is necessarily nonzero for {\it every} simple root $\beta _i$.  But if $\tilde{\alpha} \in \Delta(\tilde{\gg}^+)$, then the coefficient of $\beta _i \in \Delta _0 (V^1)$ is necessarily zero.  That is, $\tilde{\alpha} \in \Delta(\tilde{\gg}^+)$ implies $\Delta _0 (V^1) = \emptyset$ (which is equivalent to $V^1 = 0$).  Likewise if $\tilde{\alpha} \in \Delta(V^+)$, then $\Delta _0(\tilde{\gg}^1) = \emptyset$.  So one of $\Delta _0 (\tilde{\gg}^1$) or $\Delta _0 (V^1)$ must be empty.  By hypothesis $[\gg^1, \gg^2] = \gg^3 \neq 0$;  this implies $0 \neq [\gg^3, \gg^{-2}] \subset \tilde{\gg}^1$.
     \vspth \end{proof}
    
    The final  step using $\ell(\gg)\geqq 3$ is the
     \vspth \begin{Lem}\lab{4.10}
    $\ell(\gg)\geqq 3\implies [\gg^{-2},\gg^3]\subset \hh \cap \gg^1$ is non-zero.
   \vspth   \end{Lem}
    \begin{proof}
Since $\gg^1$ generates $\gg^+$ and $\gg^1$ is spanned by simple positive root vectors, there exist simple positive roots $\beta_1,\beta_2,\beta_3$ such that $\beta_1+\beta_2+\beta_3$ is a root.  Then $\gg_{-\beta_1-\beta_2}\in \gg^{-2}$ and
  \vspth\[
\Big[\gg_{\beta_1+\beta_2+\beta_3}, \gg_{-\beta_1-\beta_2}\Big] = \gg_{\beta_3}\ne 0.\qedhere\]
     \end{proof}
     
     (iii) \emph{Discussion of assumption of \eqref{4.6}.}
     The fact that we may assume \eqref{4.6} uses the following result from Proposition 3.10 in \cite{R},\footnote{The instances of $\mathbb{Q}$ in loc. cit. are typos;  the $\mathbb{Q}$ should be $\mathbb{R}$.}
     
     \begin{Thm} \lab{4.11}
     Let $G$ be a \mtg\ and $D_G=G(\R)/H_0$ a \mtd.  Any point $o\in D_G$ defines a weight zero Hodge structure on the Lie algebra $\gg$ of $G$.  Let $\wt \gg\subset \gg$ be the real semi-simple Lie subalgebra generated by $\gg^{-1,1}\oplus \gg^{1,-1}$.  Then $\wt \gg\subset \gg$ is a real sub-Hodge structure.  Any connected integral manifold of the horizontal sub-bundle corresponding to $\gg^{-1,1}\subset T_0 D_G$ is contained in $D_{\wt G}=\wt G(\R)o$ where $\wt G(\R)$ is the connected real Lie group with Lie algebra $\wt \gg$.  The horizontal sub-bundle $T^h D_{\wt G}\subset TD_{\wt G}$ is bracket generating; equivalently, $\wt \gg^{-1,1}$ generates $ \wt\gg^{-}$ under Lie bracket.
     \end{Thm}

     The general picture is that $\wt \gg$ generates an integrable sub-bundle of $TD_G$ and the $D_{\wt G}=\wt G(\R)_o$ above is a leaf of the corresponding foliation.
     
     Using this theorem the assumption \eqref{4.6} may be dropped thus completing the proof of Theorem  \ref{mainth}.\hfill\qed\medbreak
     
     The result \ref{4.11}, specifically the last sentence, is the key to where the integrability conditions imposed by transversality kick in to give atypicality of Hodge loci.
     
     We also note that in   \cite{R} it is shown that for the grading element $E\in \hh$ defined by condition $[E,X]=pX$ for all $X\in \hh^{p,-p}$ (thus $E\in$ center of $\hh^{0,0}$)
     \begin{equation}\lab{4.12} 
     \ell(\hh) = \wt\alpha (E)\end{equation}
     where $\wt\alpha$ is the highest root.

Using this we can see that it is \emph{not} always the case that $\hh^{-1,1}$ generates $\hh^{-,+}$.  To construct examples   let
     \[
     \vp :S^1 \to H(\R)\]
     be the circle defining the complex structure on $D_H$.  Then for the $E$ defined above, we can choose a Cartan subalgebra $\tt\subset \hh_\C$ and a set $\Delta\subset \tt$ of simple roots such that $E\in\tt$ and $0\leqq \alpha(E)\in \Z$ for all $\alpha\in\Delta$.  Then it can be shown (loc.\ cit.)\ that
     \begin{equation}\lab{4.14}
     \hh^{1,-1}\hensp{generates} \hh^{+,-}\iff \alpha(E)\in\{0,1\},\qquad \alpha\in \Delta.\end{equation}

    Denote by $E_\gg$ the grading element for $\gg$.  Then from \eqref{4.12} we have
    \begin{equation}\lab{4.13}
         \ell(\gg)\leqq 2\iff \wt \alpha(E_\gg)\leqq 2\end{equation}
         where $\wt\alpha$ is the highest root.   
        Thus, \emph{if} we know the  Mumford-Tate Lie algebra $\gg$, then this provides a   test for when the main theorem applies.
         
         \section{Reprise}\label{sec5}
         (i) In the second example above if  we assume that the algebra $\Hg^\bull(X)_\prim$ corresponding to $H^4(X)_\prim$ is generated by $Q$ and $\la$, then
         \[
         G_\la := \{g\in \Aut(V,Q):g\la = \la\}\]
         is the \mtg.  In \eqref{3.12} we take $H=G_\la$, $\Ga_\la=\Ga\cap G_\la$ and set $D_H=D_\la\subset D$.  Then it can be shown that 
         \setcounter{equation}{0}
         \begin{equation}
\lab{5.1}
\codim_{\Ga\bsl D} \lrp{\Phi(B)\cap (\Ga_\la\bsl D_\la)} = \codim_{\Ga\bsl D}(\Ga_\la\bsl D_\la)+\codim_{\Ga\bsl D} \Phi(B)-\dim \sig(\la).\end{equation}
Thus $\dim\sig(\la)$ is the correction term needed to convert the inequality \eqref{3.12} into an equality.
\medbreak

(ii)
Referring to the first example, for $d=\deg X\geqq 5$ 
\begin{equation}\lab{5.2}
d-3 = \codim \NL_\la<h^{2,0}(X)\end{equation}
holds only for $X$'s containing a line $\La$.  In this case the strict inequality holds for geometric, not Hodge theoretic, reasons.  Namely, if $\om\in H^0(\Om^2_X)$ has divisor $(\om) \supset \La$ containing the line, then for any $\theta\in T$ we have
\[
\lra{\om,\theta\cdot\la}=0\]
due to $\theta\cdot \om\big|_\la=0$ at the form level.

This phenomenon is general; we may say that the correction term needed to have equality in \eqref{3.13} is always positive if $\gg^3\ne (0)$, and in particular cases it may be greater than it is for a general Hodge locus in $B$ due to geometric reasons peculiar to the particular Hodge locus.

These considerations raise the following

{\setbox0\hbox{(1)}\leftmargini=\wd0 \advance\leftmargini\labelsep
 \begin{itemize}
\item[{\bf Q}:] Given $(V,Q,F^\bull)$ with \mtg\ $G$ where $\gg^3 \ne 0$, is there a \emph{uniform bound} depending only $\gg$ for the correction term needed to convert \eqref{3.12} into an equality?
\end{itemize}
}

         (iii) An example where $\gg^3 =(0)$ is given by a variation of \phs s of the form $\V_b=\V'_b\otimes \V''_b$ where each of the factors has weight 2.  As discussed above the condition $\gg^3=0$ is a non-linear PDE and the general Cartan theory of exterior differential systems gives in principle an algorithmic process for determining all local solutions.  The   example given just above is the only one we know in weight $n=4$ where $\gg^{-1,1}\ne 0$ and $\gg^3=0$.

         (iv) \emph{Hypersurface example.}  Let $X\subset \P^{n+1}$ be a smooth degree $d$ hypersurface given by an equation
         \begin{equation}\lab{5.3}
         F(x)=0\end{equation}
         where $x=[x_0,\dots,x_{n+1}]$ and $F(x)$ is homogeneous of degree $d$.  For  
         \[
         \bcs
         S^\bul= \C[x_0,\cdots,x_{n+1}],\\
         J^\bul_F= \text{Jacobian ideal }\{F_{x_0},\cdots,F_{x_{n+1}}\},\\
         R^\bul=S^\bul/J^\bul_F\ecs\]
         it is well known (\cite{G3}) that there is an isomorphism
         \[
         H^{p,n-p}(X)_\prim\cong R^{(n-p)d+n-2}.\]
         Moreover the tangent space to the family of $X$'s modulo projective equivalence is
         \[T\cong R^d\]
         and the maps
         \[
         T\otimes H^{p,n-p}(X)_\prim \to H^{p-1,n-p+1}(X)_\prim \leqno{({\rm V}.4)_p}\]
         are given by multiplication of polynomials
         \[
         R^d \otimes R^{(n-p)d+n-2}\to R^{(n-p+1)d+n-2}.\leqno{({\rm V}.5)_p}\]
         Finally since $X$ is non-singular, it follows from Macaulay's theorem that
         \setcounter{equation}{5}
         \begin{equation}\lab{5.6}
         \bmpt{4.25}{\emph{the mappings ${({\rm V}.5)_p}$ are non-zero whenever both sides are non-zero}}\end{equation}
         (cf.\ \cite{G3}).
         
         If $G$ is the \mtg\ for the period mapping of $X$'s as above, then we have 
         \[
         R^d\to \gg^{-1,1}\subset F^{-1}\End(V,Q).\]
         The image of this map is an \emph{abelian} sub-algebra $\AA\subset \gg^{-1,1}\subset \gg$.  There is an induced map 
         \begin{equation}\lab{5.7}
         \Sym^k\AA\to\gg^{-k,k}.\end{equation}
        This then gives
                  \[
         \Sym^k R^d\to \gg^{-k,k}\subset \oplus \Hom\lrp{H^{p,n-p}(X)_\prim,H^{n-k,n-p+k}(X)_\prim}\]
         which is just the map
         \[
         \Sym^k R^d \otimes R^{(n-p)d+n-2} \to R^{(n-p+k)d+n-2}\]
         given by multiplication of polynomials.  From \eqref{5.6} we may conclude that the map is non-zero whenever both sides are non-zero, which then gives that $\gg^{-3,3}\ne 0$ for $n\geqq 3, d>3$.\hfill\qed

         \medbreak (iv) \emph{The coupling length:} This is defined by $\zeta(\AA)=\max\{m:\Sym^m\AA\allowbreak\to \Hom(\V^{n,0}_b,\V^{n-m,m}_b)$ is $\ne 0\}$ at a general point of $b$.  The same argument as in the hypersurfaces example then gives \eqref{1.3} above.  There are many examples where this holds.
         
         \medbreak (v) The main Theorem III.14 from \cite{BKU} is closely related to the Ax-Schanuel conjecture, which has been established in \cite{BT}.  To state the result we keep the above notations and denote by $\CD$ the compact dual to the \mtd\ $D$.  Then $\CD=G(\C)/P$ is a homogeneous rational variety.
         
         \begin{Thm}[\cite{BT}] Let $W\subset B\times \CD$ be an algebraic subvariety and $U$ and irreducible component of $W\cap B\times_{\Ga\bsl D}\CD$  such that
         \[
         \codim_{B\times \CD}(U)<\codim_{B\times \CD}(W)+\codim_{B\times \CD}(B\times_{\Ga\bsl D}\CD).\]
         Then the projection of $U$ to $B$ is contained in a proper \mt\ subvariety.\end{Thm}
         
         In the geometric case this says that the analytic locus in $B$ where the periods satisfy a given set of algebraic relations must be of the expected codimension unless there is a reduction in the generic \mtg.
         
         (vi) Finally we will comment on the proof in \cite{BBT} that the image  $P$ of a period mapping
         \[
         \Phi:B\to \Ga\bsl D\]
         is a quasi-projective algebraic variety.
         \beb
         \item In the classical case there is the Satake-Baily-Borel completion $\ol{\Ga\bsl D}_{\rm SBB}$ of $\Ga\bsl D$.  It is a projective algebraic variety, and Borel's extension theorem gives that (after possibly passing to a finite branched cover of $\olb$)   $\Phi$ extends to a morphism
         \[
         \ol\Phi:\ol B\to \ol{\Ga\bsl D}_{\rm SBB}\]
         of algebraic varieties.  This gives the result.
           \item  In the non-classical case  the Hodge variety $\Ga\bsl D$ is not an algebraic variety; on it the only global meromorphic functions are constant.
       \item A this juncture the new concept of a \emph{definable $o$-minimal structure} enters the picture.  A non-compact analytic variety can have such a structure only if it is ``tame at infinity" (cf.\ \cite{BKT}).
       \item For definable $o$-minimal maps of algebraic varieties to analytic varieties having a definable $o$-minimal structure there are algebraicity results; among other things    the image is algebraic.
       \item The by now classical analysis of degenerations of Hodge structures (\cite{S}, \cite{CKS}, \cite{Ka}) is then used to show that the images of period mappings have a definable $o$-minimal structure.
       \item Finally the line bundle $L := \ottimes^p \det F^p$ may be shown to induce an ample line bundle on the image $\Phi(B)\subset \Ga\bsl D$, which is then a quasi-projective algebraic variety.\footnote{It has recently been proved \cite{BFMT} that the canonical extension $L_e \rightarrow \bar{B}$ of $L \rightarrow B$ is semi-ample;  thus Proj($L_e$) defines a canonical projective completion of $\Phi(B)$ giving an extension of the classical Baily-Borel-Satake compactification of the image of period mappings.}
               \eeb
               The use of definable $o$-minimal structures has introduced a whole new set of ideas into Hodge theory and into the arithmetic properties of a variation of Hodge structure and period mappings (\cite{K1}).
    \bibliographystyle{amsalpha}
     \bibliography{AtypicalHodgeLoci}
      \end{spacing}
      \end{document}

%% file: pgmacs.tex
 \usepackage{graphicx}
  \usepackage{epstopdf}
     \usepackage{amssymb}
  \usepackage{color} 

\long\def\symbolfootnote[#1]#2{\begingroup%
\def\thefootnote{\fnsymbol{footnote}}\footnote[#1]{#2}\endgroup}

\newcommand\CD{\check{D}}

\newcommand\fbul{F^\bullet}

\newcommand\bsm{ \begin{smallmatrix}}
\newcommand\bspm{ \left(\begin{smallmatrix}}
\newcommand\esm{\end{smallmatrix} }
\newcommand\espm{\end{smallmatrix} \right)}
\newcommand\bbm{\left[\begin{matrix}}
\newcommand\ebm{\end{matrix}\right]}
\newcommand\bcs{\begin{cases}}
\newcommand\ecs{\end{cases}}

 \newcommand{\lra}[1]{\left\langle#1\right\rangle}

\newcommand{\lrp}[1]{\left(#1\right)}

 \newcommand\wt[1]{\widetilde{#1}}

 \newcommand\V{\mathbb{V}}
 
 \renewcommand\L{\mathbb{L}}

\newcommand\hh{\mathfrak{h}}

\renewcommand\tt{\mathfrak{t}}

\newcommand{\C}{{\mathbb{C}}}

\renewcommand{\P}{\mathbb{P}}
\newcommand{\Q}{{\mathbb{Q}}}

\newcommand{\R}{\mathbb{R}}

\newcommand{\Z}{\mathbb{Z}}

\newcommand{\cF}{{\mathscr{F}}}

\newcommand{\cL}{{\mathscr{L}}}

\newcommand{\cO}{{\mathscr{O}}}

\newcommand{\cX}{{\mathscr{X}}}

\newcommand\bpm{\begin{pmatrix}}
\newcommand\epm{\end{pmatrix}}

\newcommand\mtd{Mumford-Tate domain}

\newcommand\mtg{Mumford-Tate group}
\newcommand\phs{polarized Hodge structure}

\newcommand\mt{Mumford-Tate}

\newcommand\NL{{\mathop{\rm NL}\nolimits}}

\newcommand\MT{\mathop{\rm MT}\nolimits}

\newcommand{\Hg}{{\rm Hg}}

 \newcommand\Aut{\mathop{\rm Aut}\nolimits}
\newcommand\End{\mathop{\rm End}\nolimits}

\newcommand\GL{\mathop{\rm GL}\nolimits}

\newcommand\Sym{\mathop{\rm Sym}\nolimits}

\newcommand\prim{{\rm prim}}

\newcommand{\codim}{\mathop{\rm codim}\nolimits}

 \newcommand{\Hom}{\mathop{\rm Hom}\nolimits}

 \renewcommand{\part}{\partial}
\newcommand{\la}{{\lambda}}

\newcommand\Ga{{\Gamma}}
\newcommand{\La}{{\Lambda}}
\newcommand{\Om}{{\Omega}}
\newcommand{\om}{{\omega}}

\newcommand\vp{\varphi}
\newcommand\sig{\sigma}

\newcommand\bsl{\backslash}

 \renewcommand{\AA}{{\mathfrak{A}}}

\newcommand{\lab}{\label}

  \renewcommand\theenumi{\roman{enumi}}

\newcommand{\hensp}[1]{\enspace\hbox{#1}\enspace}

\newcommand{\opplus}{\mathop{\oplus}\limits}
\newcommand{\ottimes}{\mathop{\otimes}\limits}

 \newcommand\bmpt[2]{\hbox{\begin{minipage}[t]{#1in}  #2 \end{minipage}}}
 \renewcommand\thesubsection{\thesection.\arabic{subsection}}

\newcounter{demo}[equation]

 \newtheoremstyle{mytheo}
  {3pt}
  {3pt}
  {\itshape}
  {}
  {\scshape}
  {:}
  {.5em}
  {}

\theoremstyle{mytheo}

\newtheorem{Thm}[equation]{Theorem}

\newtheorem{Lem}[equation]{Lemma}

\newtheorem{Prop}[equation]{Proposition}

 \newtheoremstyle{subsect}
 {3pt}
  {3pt}
  {}
  {}
  {\it}
  {\upshape{:}}
  {.5em}
  {}
\theoremstyle{subsect}

\newtheoremstyle{note}
  {3pt}
  {3pt}
  {}
  {}
  {\bfseries}
  {:}
  {.5em}
  {}
\theoremstyle{note}

\theoremstyle{remark}

\newcommand\simto{\xrightarrow{\sim}}

\newcommand\ol[1]{\overline{#1}}

\newcommand\olb{{\overline{B}}}

 \usepackage[mathscr]{euscript}

 \renewcommand\thesubsection{\thesection.\Alph{subsection}}

 \renewcommand{\theequation}{\thesubsection.\arabic{equation}} 